\documentclass[12pt]{amsart}
\usepackage{tikz}

\usepackage{amsmath,amsfonts,amsthm,amssymb,amscd, verbatim, graphicx,textcomp, hyperref}
\usepackage{upgreek}
\usepackage[margin=2cm]{geometry}
\usepackage{pinlabel}
\usepackage{lscape}
\usepackage{tabularx}
\usepackage{latexsym}
\usepackage{multirow}
\usepackage{mathtools}
\usepackage{float}
 \usepackage{color}

\newtheorem{Thm}{Theorem}[section]

\newtheorem{Prop}[Thm]{Proposition}

\newtheorem{Lem}[Thm]{Lemma}

\def \ov {\overline}

\theoremstyle{definition}

\theoremstyle{remark}

\numberwithin{equation}{section}

\newcommand{\Aut}{\operatorname{Aut}}

\newcommand{\Hom}{\operatorname{Hom}}

\newcommand{\Syl}{\operatorname{Syl}}

\newcommand{\Out}{\operatorname{Out}}

\newcommand{\Inn}{\operatorname{Inn}}

\newcommand{\PSL}{\operatorname{PSL}}
\newcommand{\GL}{\operatorname{GL}}

\newcommand{\SL}{\operatorname{SL}}

\renewcommand{\epsilon}{\varepsilon}

\newcommand{\F}{\mathcal{F}}

\def \ov {\overline}

\begin{document}

\title{Fusion systems on maximal class $3$-groups of rank two revisited}

\author{Chris Parker}
\address{
School of Mathematics\\
University of Birmingham\\
Edgbaston\\
Birmingham B15 2TT\\ United Kingdom
} \email{c.w.parker@bham.ac.uk}

\author{Jason Semeraro}
\address{Heilbronn Institute for Mathematical Research, Department of Mathematics, University of Bristol,  United Kingdom}
\email{js13525@bristol.ac.uk}

\begin{abstract}
We complete the determination of saturated fusion systems on maximal class $3$-groups of rank two.
\end{abstract}

\maketitle
\section{Introduction}\label{s:intro}

The maximal class $3$-groups have been classified by Blackburn in \cite{blackburn1958special}. In this article we revisit the determination of the saturated fusion systems on these groups.  We take the presentations for the maximal class $3$-groups from \cite{DiazRuizViruel2007}. For $r \ge 4$, and $\beta, \gamma, \delta \in \{0,1,2\}$, define $$B(r; \beta, \gamma, \delta)=\langle s, s_1,\dots,s_{r-1}\mid \textbf{R1},\textbf{R2}, \textbf{R3}, \textbf{R4}, \textbf{R5}, \textbf{R6}\rangle$$ where the relations are as follows:
\begin{enumerate}
\item[\textbf{R1}:] $s_i=[s_{i-1},s]$ for $i\in\{2, \dots, r-1\}$;
\item [\textbf{R2}:] $[s_1,s_i]=1$ for $i\in\{3, \dots, r-1\}$;
\item [\textbf{R3}:] $s_i^3s_{i+1}^3s_{i+2} = 1$ for $i \in \{2, \dots, r-1\}$ where $s_r=s_{r+1}=1$ by definition;
\item [\textbf{R4}:] $[s_1,s_2]=s_{r-1}^\beta$;
\item [\textbf{R5}:] $s_1^3s_2^3s_3=s_{r-1}^\gamma$; and
\item [\textbf{R6}:] $s^3=s_{r-1}^\delta$.
\end{enumerate}

\noindent We mostly require that $r \ge 5$: note that $|B(r;\beta,\gamma,\delta)|=3^r$ and that there are isomorphisms between some of the groups listed. The full list of maximal class $3$-groups of order at least $3^5$ is uniquely given up to isomorphism by the requirements:
\begin{enumerate}
\item for $r$ odd, $$(\beta,\gamma, \delta) \in \{(1,0,0), (1,0,1), (1,0,2), (0,1,0),  (0,0,1), (0,0,0)\}.$$
\item for $r$ even, $$(\beta,\gamma, \delta) \in \{(1,0,0), (1,0,1), (1,0,2), (0,1,0),  (0,0,1), (0,0,0), (0,2,0)\}.$$
\end{enumerate}
Thus, when $r \ge 5$, there are  six maximal class $3$-groups when $r$ is odd and  seven when $r$ is even.

Recall that for a prime $p$ a saturated fusion system $\F$ on a $p$-group $S$ is \textit{reduced} if and only if $O^p(\F)=O^{p'}(\F)=\F$ and $O_p(\F)=1$. The fusion system $\F$ is \emph{exotic} if $\F \ne \F_S(G)$ for all finite groups $G$ with $S\in \Syl_p(G)$. Our main result is as follows.

\begin{Thm}\label{Main} Suppose that $S= B(r;\beta,\gamma,\delta)$ is a maximal class $3$-group of order at least $3^5$.  Assume that $\F$ is a saturated  fusion system on $S$ and that $\F$ has at least one $\F$-conjugacy class of $\F$-essential subgroups. Then either $\F$ is as described in \cite[Theorem 5.10]{DiazRuizViruel2007} or $\beta \ne 0$ and one of the following holds:
\begin{enumerate}
\item [(i)]$S=B(r;1,0,0)$,  $\langle s, s_{r-1}\rangle $ represents the unique $\F$-conjugacy class of $\F$-essential subgroups, $\Aut_\F(\langle s, s_{r-1}\rangle) \cong \SL_2(3)$, $|\Out_\F(S)|=2$ and either
    \begin{enumerate}
    \item[(a)] $r$ is even and $\F$ is reduced; or
    \item[(b)] $r=2k+1$ is odd, and $O^3(\F)$ is a subsystem of index $3$ in $\F$ isomorphic to the fusion system of $\PSL_3(q)$ at the prime $3$ for some prime power $q$ with $v_3(q-1)=k$.
        \end{enumerate}
\item[(ii)] $S=B(r;1,0,2)$,   $r$ is even and one of the following holds:
\begin{enumerate}
 \item[(a)] $\langle ss_1, s_{r-1}\rangle $  represents the unique  $\F$-conjugacy class of $\F$-essential subgroups, $\Aut_\F(\langle ss_1, s_{r-1}\rangle ) \cong \SL_2(3)$, $|\Out_\F(S)|=2$ and $\F$ is reduced;
 \item [(b)]$\langle ss_1^2, s_{r-1}\rangle $ represents the unique  $\F$-conjugacy class of $\F$-essential subgroups, $\Aut_\F(\langle ss_1^2, s_{r-1}\rangle ) \cong \SL_2(3)$, $|\Out_\F(S)|=2$ and $\F$ is reduced; or
 \item [(c)] there are two  $\F$-conjugacy classes of $\F$-essential subgroups represented by $\langle ss_1, s_{r-1}\rangle $ and $\langle ss_1^2, s_{r-1}\rangle $ with $\Aut_\F(\langle ss_1^2, s_{r-1}\rangle ) \cong \Aut_\F(\langle ss_1, s_{r-1}\rangle ) \cong \SL_2(3)$, $|\Out_\F(S)|=2$ and $\F$ is reduced.
\end{enumerate}\end{enumerate}
Furthermore, the fusion systems listed in (i) and (ii) are exotic.
 \end{Thm}  The smallest fusion systems listed in Theorem~\ref{Main} appeared  during the work of the authors to classify all reduced fusion systems on groups of order up to $1000$ by computer \cite{ps2018}. We note that all of them have $\F$-pearls \cite{grazian2018fusion}.

Inspection of the proof of \cite[Theorem 1.1]{DiazRuizViruel2007} reveals two distinct misstatements which lead to  the absence of the fusion systems in parts (i) and (ii) of Theorem~\ref{Main} from their results.  The first is in the assertion that every element of $S\setminus \gamma_1(S)$ has order $9$ in the case that $\delta\ne 0$ (\cite[Proposition A.9 (e)]{DiazRuizViruel2007}). The groups $B(r; 1,0,2)$ disprove this claim (as does \cite[Proposition A.9(a)]{DiazRuizViruel2007}.) This former ``fact" is used in the proof of \cite[Theorem 5.1]{DiazRuizViruel2007} where it is claimed that the groups $B(r;\beta,\gamma,\delta)$ are resistant when $\delta \ne 0$ and this is false.   The second misstatement  is in the description of the automorphism groups of the groups $B(r;\beta,\gamma,0)$ in \cite[Lemma A.14]{DiazRuizViruel2007}. Specifically, the group $B(r;1,\gamma,0)$ has automorphisms with $e=\pm 1$ in the notation of that result whereas it is claimed that $e=1$.  This leads the authors to work under the assumption that $\beta=0$  in the proof of \cite[Theorem 5.8]{DiazRuizViruel2007}. In the case that $\beta=0$ we are convinced that their calculations are accurate.  In this note we determine automorphism groups of $B(r;\beta,\gamma,\delta)$ for all maximal class $3$-groups of order at least $3^5$ in Proposition~\ref{autogrpsdesc}  and we hope this may be of some independent value. We have also included explicit calculations of various other facts that we could perhaps have cited to other sources.  For example, an expression   similar to that given in Proposition \ref{autogrps} (iii) appears in the discussion which precedes \cite[Proposition 3.4]{mazza2008connected}.

The results from \cite{DiazRuizViruel2007} have 26 MathSciNet citations (August 2018). Many of these use facts about the groups $B(r;\beta,\gamma,\delta)$.   For example Mazza \cite{mazza2008connected} uses \cite[Proposition A.9 (a)]{DiazRuizViruel2007} to calculate correctly that there are elements of order 3 in  the group $S=B(r; 1,0,2)$ which are not in $\gamma_1(S)$ (we give an explicit proof of \cite[Proposition A.9 (a)]{DiazRuizViruel2007} in Section \ref{s:the calc}). Other citations use \cite{DiazRuizViruel2007} to present examples of exotic fusion systems or $3$-local compact groups    and so they are not impacted by the existence of further saturated fusion systems.    There are two true uses of the classification of these fusion systems that we know of. The first
is in a paper of Sambale \cite{sambale2013further} where the classification is used to verify a conjecture of Olsson for certain 3-blocks with maximal class defect groups. In an e-mail exchange, Sambale has explained that his result also holds when the new fusion systems are taken into account.
 The second is by Malle,  Navarro and Sambale in \cite{malle2017defects} where they propose upper bounds for the number of modular constituents of the reduction modulo $p$ of a complex irreducible character of a finite group.  As far as we can tell the existence of the fusion systems in (i) and (ii) of Theorem~\ref{Main} do not change the results in \cite{malle2017defects} as they do not have $\F$-essential subgroups of order $27$.

  We assume familiarity with the background needed to work with fusion systems and just refer to \cite{AschbacherKessarOliver2011, CravenTheory} as our sources.

\section{Maximal class $3$-groups of order at least $3^5$}

 Let $S=B(r;\beta,\gamma,\delta)$ be a maximal class $3$-group with $r \ge 5$ as described in Section \ref{s:intro}. We set $$\gamma_1(S):= \langle s_1,\dots, s_{r-1}\rangle$$ and, for $i >1$,  $$\gamma_i(S)= [\gamma_{i-1}(S),S].$$
As $S$ has maximal class, using $\textbf{R2}$ we obtain
\begin{Lem}\label{gammai}
For $1\le i \le r-1$, $\gamma_i(S)= \langle s_i, \dots, s_{r-1}\rangle$ and $S>\gamma_2(S)>\dots> \gamma_{r-1}(S)$ is the lower central series of $S$. In particular, $|\gamma_{r-1}(S)|=|Z(S)|=o(s_{r-1})=3$.\qed
\end{Lem}

\begin{Lem}\label{gamma1struct}
 Either $\beta=0$ and $\gamma_1(S)$ is abelian or $\beta\ne 0$, $\gamma_1(S)$  has centre $\gamma_3(S)$ and derived group $\gamma_{r-1}(S)$. In particular, $\gamma_2(S)$ is abelian and $\gamma_1(S)= C_S(\gamma_{r-2}(S))=C_S(\gamma_2(S)/\gamma_4(S))$ is characteristic in $S$.
\end{Lem}

\begin{proof}  Since $r \ge 5$, \textbf{R2} implies  $s_1$ centralizes $\gamma_{r-2}(S)=\langle s_{r-2},s_{r-1}\rangle$ which is abelian. Since $C_S(\gamma_{r-2}(S))$ is normal in $S$ and $s_1 \in C_S(\gamma_{r-2}(S))$, the fact that $S$ has maximal class implies that $ C_S(\gamma_{r-2}(S))=\gamma_1(S)$. Hence $\gamma_{r-2}(S) \le Z(\gamma_1(S))$. Assume that $\gamma_{e}(S) \le Z(\gamma_1(S))$ for some $3< e \le r-2$. Then  $s_{e-1}$ centralizes $\gamma_{e}(S)$ and so $\gamma_{e-1}(S)$ is abelian and \textbf{R2} implies $s_1$ centralizes $\gamma_{e-1}(S)$.  Since $S$ has maximal class, this implies $C_S(\gamma_{e-1}(S))=\gamma_1(S)$. Thus $\gamma_{e-1}(S)\le Z(\gamma_1(S))$ and we conclude that $\gamma_{3}(S) \le Z(\gamma_1(S))$ by induction.  Now $\gamma_2(S)$ is abelian and so if $s_1$ and $s_2$ commute, then $\gamma_1(S)$ is abelian, whereas if $s_1$ and $s_2$ do not commute, then $Z(\gamma_1(S))= \gamma_3(S)$, $\gamma_2(S)$ is abelian and $\gamma_1(S)'= \gamma_{r-1}(S)$. This proves the claim.
\end{proof}

\begin{Lem}\label{2gen}
We have $\Omega_1(\gamma_1(S))= \langle s_{r-1},s_{r-2}\rangle$. In particular, every subgroup of $\gamma_1(S)$ is $2$-generated.
\end{Lem}

\begin{proof}
By  \textbf{R3}, we have $\Omega_1(\gamma_1(S))\ge \langle s_{r-1},s_{r-2}\rangle$. Assume that $\Omega_1(\gamma_1(S)) \ne  \langle s_{r-1},s_{r-2}\rangle$.
 If $\gamma_1(S)$ is abelian, then $\Omega_1(\gamma_1(S))$ has exponent $3$ and, as $\gamma_1(S)$ is characteristic in $S$, $\Omega_1(\gamma_1(S)) \ge \gamma_{r-3}(S)$. However,  \textbf{R3} shows that $s_{r-3}^3= s_{r-1}^2$, a contradiction.  Hence, if $\gamma_1(S)$ is abelian, the result holds. Furthermore, in this case we have that every subgroup of $\gamma_1(S)$ is $2$-generated.

 Suppose that $\gamma_1(S)$ is non-abelian. Then the derived subgroup of $\gamma_1(S) $ is $\gamma_{r-1}(S)$ by Lemma~\ref{gamma1struct}. Notice that $\ov S=S/\gamma_{r-1}(S) \cong B(r-1;0,0,0)$. Thus $\Omega_1(\gamma_1(S))/\gamma_{r-1}(S)\le \langle s_{r-3}\gamma_{r-1}(S),s_{r-2}\gamma_{r-1}(S) \rangle $ by applying the previous case to $\ov S$. This shows that $\Omega_1(\gamma_1(S)) \le \langle s_{r-3},s_{r-2},s_{r-1} \rangle$ and as $r \ge 5$, again we see that $s_{r-3}$ has order $9$ by  \textbf{R3} and we conclude that  $\Omega_1(\gamma_1(S)) = \langle s_{r-1},s_{r-2}\rangle$ in this case also.

Finally, assume that $A \le \gamma_1(S)$ is at least $3$-generated. If $A$ does not contain $\gamma_{r-1}(S)$, then $A$ is isomorphic to a subgroup of $\gamma_1(S)/\gamma_{r-1}(S)$ and so is $2$-generated, a contradiction.  Hence $\gamma_{r-1}(S)\le A$.  Furthermore, $A/\gamma_{r-1}(S)$ is $2$-generated and $\Omega_1(A/\gamma_{r-1}(S))=\langle s_{r-2}, s_{r-3}\rangle \gamma_{r-1}(S)$. Since $s_{r-3}^3 =s_{r-1}^2$, we have $A$ is $2$-generated a contradiction.
\end{proof}

Because of Lemmas \ref{gamma1struct} and \ref{2gen}, we have $\Omega_1(\gamma_1(S)/\Omega_1(S))$ has order $3^2$ when $r \ge 5$ and since $\Omega_i(\gamma_1(S))/\Omega_1(S)=\Omega_{i-1}(\gamma_1(S)/\Omega_1(S))$, by induction $\Omega_i(\gamma_1(S))$ has order at most $3^{2i}$ for each $1 \le i \le \lfloor \frac{r}{2} \rfloor.$ In particular, $\gamma_1(S)$ has exponent at most  $\lfloor \frac{r}{2} \rfloor.$

\begin{Lem}\label{conjugate in coset} If  $x \in S\setminus \gamma_1(S)$, then $C_S(x)= \langle x, s_{r-1}\rangle$ has order $9$ and all the elements of the coset  $x\gamma_2(S)$ are $S$-conjugate.
\end{Lem}

\begin{proof}
Suppose that $x \in S \setminus \gamma_1(S)$. Then $C_S(x)= \langle x \rangle C_{\gamma_1(S)}(x)$. Obviously, $\gamma_{r-1}(S) \le C_{\gamma_1(S)}(x)$ and we know from Lemma~\ref{gamma1struct} that $\gamma_1(S)'\le \gamma_{r-1}(S)$. This means that $C_{\gamma_1(S)}(x)$ is normal in $\langle x,\gamma_1(S)\rangle = S$. Assume that $C_{\gamma_1(S)}(x)> \gamma_{r-1}(S)$. Then $s_{r-2}\in  C_{\gamma_1(S)}(x)$. As $r \ge 5$, $r-2 \ge 3$ and so $s_{r-2}$ is centralized by $\langle x, s_1\rangle = S$, and this contradicts $Z(S)=\gamma_{r-1}(S)$. It follows that $C_S(x)=\langle x, s_{r-1}\rangle$ has order $9$ and that  $|x^S|=3^{r-2}= |x\gamma_2(S)|$. This proves the result.
\end{proof}

 \section{Automorphisms of maximal class $3$-groups}\label{s:the calc}

We continue to assume that $S=B(r;\beta,\gamma,\delta)$ with $r \ge 5$.
We will repeatedly use the commutator formulae $$[xy,z]=[x,z]^y[y,z]=[x,z][x,z,y][y,z]$$ and $$[x,yz]=[x,z][x,y]^z=[x,z][x,y][x,y,z]$$ without mention. In particular, we note the following consequence of these relations in our situation. Suppose that $a,b  \in S$, and $v,w \in \gamma_2(S)$.  Then, as $\gamma_2(S)$ is abelian by Lemma~\ref{gamma1struct}, we have $$[av,bw]=[a,w][a,b][v,b].$$

\begin{Lem}\label{relations} Suppose that $S= B(r;\beta,\gamma, \delta)$, let $d \in \{0,1,2\}$ and $e,f \in \{1,2\}$. The following identities hold.
\begin{itemize}
\item[(i)]
$w^3[w,s]^3[w,s,s]=w^3[w,s^2]^3[w,s^2,s^2]=1$ for all $w \in \gamma_2(S)$.
\item[(ii)] $[[s_1^f,s^e],s_1^d]= s_{r-1}^{2def\beta}$.
\item[(iii)] $(s_1^f)^3[s_1^f,s^e]^3[s_1^f,s^e,s^e]=s_{r-1}^{f\gamma}.$

\end{itemize}

\end{Lem}

\begin{proof}
Write $w= s_2^{e_2}\dots s_{r-1}^{e_{r-1}}$ for suitable $e_j \in \{0,1,2\}$.
Then, as $\gamma_2(S)$ is abelian, $$[w,s]= \prod_{j=2}^{r-1} [s_j,s]^{e_j} = \prod_{j=2} ^{r-1} s_{j+1}^{e_j}$$ and
 $$[w,s,s]= \prod_{j=2}^{r-1} [s_{j+1},s]^{e_j} = \prod_{j=2} ^{r-1} s_{j+2}^{e_j}.$$
 Therefore

\begin{eqnarray}\label{neweq}w^3[w,s]^3[w,s,s] &=& \prod_{j=2}^{r-1} s_j^{3e_j} \prod_{j=2} ^{r-1} s_{j+1}^{3e_j}\prod_{j=2} ^{r-1} s_{j+2}^{e_j}= \prod_{j=2}^{r-1} s_j^{3e_j}  s_{j+1}^{3e_j} s_{j+2}^{e_j}=\prod_{j=2}^{r-1} (s_j^{3}  s_{j+1}^{3} s_{j+2})^{e_j}=1.\end{eqnarray} Now we calculate
 \begin{eqnarray*}
 [w,s^2]&=&[w,s]^2[w,s,s]
 \end{eqnarray*}
and, by using Equation~\ref{neweq} with $[w,s,s]$ in place of $w$ for the final equality,
  \begin{eqnarray*}
 [w,s^2,s^2]&=&[[w,s]^2[w,s,s],s^2]
\\ &=&[w,s,s^2]^2[w,s,s,s^2]\\
 &=&([w,s,s]^2[w,s,s,s])^2[w,s,s,s]^2[w,s,s,s,s]\\
 &=&[w,s,s]^4[w,s,s,s]^4[w,s,s,s,s]\\
 &=&[w,s,s][w,s,s,s].
 \end{eqnarray*}
We obtain:

\begin{eqnarray*}\label{e:3}
w^3[w,s^2]^3[w,s^2,s^2] &=& w^3([w,s]^6[w,s,s]^3)[w,s,s][w,s,s,s] \\&=&w^3[w,s]^3[w,s,s][w,s]^3[w,s,s]^3[w,s,s,s]\\ &=&1.
\end{eqnarray*}
This proves (i).
We also calculate
$$[[s_1^f,s^e],s_1^d]=\begin{cases}[s_2,s_1^d]=s_{r-1}^{2d\beta}&\mbox{ if } e=f=1;\\
[s_2^2s_{r-1}^{2\beta},s_1^d]=s_{r-1}^{d\beta}&\mbox{ if } e=1, f=2;\\
[s_2^2s_3,s_1^d]=s_{r-1}^{d\beta}&\mbox{ if } e=2, f=1;\\
[s_2^4s_3^4s_4,s_1^d]=[s_2,s_1^d]=s_{r-1}^{2d\beta}&\mbox{ if } e=f=2.
\end{cases}$$
Hence (ii) holds.

For part (iii),  the case $e=f=1$ is immediate using \textbf{R2} and \textbf{R5}.  Suppose that $e=1$ and $f=2$.  Then, as $[s_1,s,s_1] \in \gamma_1(S)' \le \gamma_{r-1}(S)$ which has order $3$, $[s_1,s,s_1]^3=1$ and so
\begin{eqnarray*}(s_1^2)^3 [s_1^2,s]^3[s_1^2,s,s]&=&s_1^6([s_1,s][s_1,s,s_1][s_1,s])^3[[s_1,s][s_1,s,s_1][s_1,s],s]\\
&=&s_1^6[s_1,s]^6[[s_1,s][s_1,s,s_1][s_1,s],s]= s_1^6s_2^6[s_2^2,s]\\
&=&  s_1^6s_2^6s_3^2 = s_1^3s_1^3s_2^3s_3s_2^3s_3= (s_1^3s_2^3s_3)^2= s_{r-1}^{2\gamma}.
\end{eqnarray*}
Now we calculate when $e=2$ and $f=1$,
\begin{eqnarray*}s_1^3 [s_1,s^2]^3[s_1,s^2,s^2]&=&s_1^3(s_2^2s_3)^3[s_2^2s_3,s^2]
= s_1^3s_2^6s_3^3 s_3^4s_4^4s_5  = s_1^3s_2^3s_3=       s_{r-1}^{\gamma}.
\end{eqnarray*}
Finally, assume that $e=f=2$.
Then

\begin{eqnarray*}(s_1^2)^3[s_1^2,s^2]^3[s_1^2,s^2,s^2]&=&s_1^6[s_1^2,s]^6[s_1^2,s,s]^3 [s_1^2,s^2,s^2]\\
&=&  s_1^6[s_1^2,s]^6[s_1^2,s,s]^3 [[s_1^2,s]^2[s_1^2,s,s],s^2]   \\
&=& s_1^6([s_1,s][s_1,s,s_1][s_1,s])^6[s_1^2,s,s]^3 [[s_1^2,s]^2[s_1^2,s,s],s^2]\\
&=& s_1^6s_2^{12}[s_1^2,s,s]^3 [[s_1^2,s]^2[s_1^2,s,s],s^2]\\
&=& s_1^6s_2^{12}[[s_1,s][s_1,s,s_1][s_1,s],s]^3 [[s_1^2,s]^2[s_1^2,s,s],s^2]\\
&=& s_1^6s_2^{12}[s_2^2,s]^3 [[s_1^2,s]^2[s_1^2,s,s],s^2]\\
&=& s_1^6s_2^{12}s_3^6[s_2^4[s_1,s,s_1]^2[s_2^2[s_1,s,s_1],s],s^2]\\
 &=& s_1^6s_2^{12}s_3^6[s_2^4[s_1,s,s_1]^2s_3^2,s^2]\\
  &=& s_1^6s_2^{12}s_3^6[s_2^4s_3^2,s^2]
   =s_1^6s_2^{12}s_3^6[s_2^4,s^2] [s_3^2,s^2]\\
   &=& s_1^6s_2^{12}s_3^6[s_2^4,s]^2[s_2^4,s,s] [s_3^2,s]^2[s_3^2,s,s]\\
    &=& s_1^6s_2^{12}s_3^6s_3^8s_4^4s_4^4 s_5^2
     = s_1^6s_2^{12}s_3^8s_4^2
      = s_1^6s_2^{6}s_3^2 =s_{r-1}^{2\gamma}.\\
\end{eqnarray*}
This establishes (iii).
 \end{proof}

For $v,w,\in \gamma_2(S)$, $d \in \{0,1,2\}$ and $e,f \in \{1,2\}$ define $\theta_{e,d,f,v,w}: S\rightarrow S$  by

\begin{center}
\renewcommand{\arraystretch}{1.4}
\begin{tabular}{llcc}
\multirow{2}{*}{$\theta_{e,d,f,v,w}:$} & $s$ &$\mapsto$ &$s^es_1^dv$ \\
                   & $s_1$&$\mapsto$ & $s_1^fw$.
\end{tabular}
\end{center}
Suppose that $\theta= \theta_{e,d,f,v,w}$ is one of these maps.
 We shall investigate the restrictions on $e,d,f,v$ and $w$ required to ensure that $\theta$ is an automorphism of $S$.

Define $t:= s\theta =s^es_1^dv$ and $t_1:=s_1\theta=s_1^fw$. For $j >1$, set $$ t_j:=[t_{j-1},t].$$
Note that for $k\ge r$, $t_k=1$  and $t_j\in \gamma_j(S)$ for all $j \ge 1$. Since $\gamma_3(S)\le Z(\gamma_1(S))$, $[t_1,t_i]=1$ for $i\in \{3,\dots,r-1\}$. Thus \textbf{R1} and  \textbf{R2} are satisfied.

    Notice that
$$[t_2,t]=[t_2,s^es_1^dv]= [t_2,s^e][t_2,s_1^dv]$$ and $$[t_2,t,t]=[t_2,s^e,s^e].$$ Since $[t_2,s_1^dv] \in \gamma_{r-1}(S)$, $[t_2,s_1^dv]^3=1$. Hence, by   Lemma~\ref{relations}(i),
$$t_2^3t_3^3t_4= t_2^3([t_2,s^e][t_2,s_1^dv])^3[t_2,s^e,s^e]=  t_2^3[t_2,s^e] ^3[t_2,s^e,s^e]=1.$$
Suppose that $j \ge 3$. Then $t_j \in \gamma_3(S)\le Z(\gamma_1(S))$ and so $[t_j,t]= [t_j,s^e]$ and $[t_j,t,t]= [t_j,s^e,s^e]$.  Therefore we can apply Lemma~\ref{relations} again to obtain $$t_j^3t_{j+1}^3t_{j+2}=1.$$
Thus \textbf{R3} is satisfied.

We start to investigate \textbf{R4}. We have $t_1= s_1^fw$, and for some integer $n_2$,
\begin{eqnarray*}t_2&=&[t_1,t]\\&=&[s_1^fw,s^es_1^dv]\\&=&[s_1^f,v][s_1^f,s^es_1^d][w,s^es_1^d]\\&=& [s_1^f,v][s_1^f,s^e][s_1^f,s^e,s_1^d][w,s^es_1^d]\\&=&[s_1^f,v][s_1^f,s^e][s_1^f,s^e,s_1^d][w,s_1^d][w,s^e][w,s^e,s_1^d]\\&=&
[s_1^f,s^e][w,s^e] s_{r-1}^{n_2}.\end{eqnarray*} Therefore,\begin{eqnarray*}\label{t3calc} t_3&=&[t_2,t]\\&=&[[s_1^f,s^e][w,s^e] s_{r-1}^{n_2},s^es_1^dv]\\&=&[[s_1^f,s^e],s_1^dv][[s_1^f,s^e],s^e][[w,s^e] s_{r-1}^{n_2},s^e]\\&=&[[s_1^f,s^e],s_1^dv][[s_1^f,s^e],s^e][w,s^e,s^e]\\&=&[[s_1^f,s^e],s_1^d][[s_1^f,s^e],s^e][w,s^e,s^e].\end{eqnarray*}
Note that  $t_2=[s_1^f,s^e][w,s^e] s_{r-1}^{n_2}=s_2^{ef}g_3$ for some $g_3 \in \gamma_3(S)$. Similarly, $t_3= s_3^{e^2f}g_4$ for some $g_4\in\gamma_4(S)$. Continuing in this manner we see that $t_j= s_j^{e^{j-1}f}g_{j+1}$ for some $g_{j+1}\in\gamma_{j+1}(S)$. In particular,
\begin{equation}\label{tr-1calc} t_{r-1}= s_{r-1}^{e^{r-2}f}. \end{equation}
Again we calculate
\begin{eqnarray*}
[t_1,t_2] &=&[s_1^fw, [s_1^f,s^e][w,s^e] s_{r-1}^{n_2}]= [s_1^fw, [s_1^f,s^e][w,s^e]]\\&=& [s_1^fw, [w,s^e]][s_1^fw, [s_1^f,s^e]][s_1^fw, [s_1^f,s^e],[w,s^e]]\\
&=&[s_1^fw, [w,s^e]][s_1^fw, [s_1^f,s^e]]\\
&=&[s_1^fw, [s_1^f,s^e]]\\
&=&[s_1^f, [s_1^f,s^e]]=s_{r-1}^{ef^2\beta}.\\
\end{eqnarray*}
where the last equality follows from Lemma \ref{relations}(ii) and the fact that $s_{r-1}$ has order $3$.

Therefore, for \textbf{R4} to hold we must have
$$ef\beta \equiv e^{r-2}\beta\pmod 3.$$
Now, recalling that $\gamma_2(S)$ is abelian and $\gamma_{r-1}(S)$ has order 3, we calculate

\begin{eqnarray*}
t_1^3t_2^3t_3&=&(s_1^fw)^3([s_1^f,s^e][w,s^e] s_{r-1}^{n_2})^3[[s_1^f,s^e],s_1^d][[s_1^f,s^e],s^e][w,s^e,s^e]\\
&=&(s_1^fw)^3[s_1^f,s^e]^3[w,s^e] ^3 [[s_1^f,s^e],s^e][w,s^e,s^e]s_{r-1}^{2def\beta}\\
&=&(s_1^f)^3w^3[w,s_1^f]^3[s_1^f,s^e]^3[w,s^e] ^3 [[s_1^f,s^e],s^e][w,s^e,s^e]s_{r-1}^{2def\beta}\\
&=&(s_1^f)^3 [s_1^f,s^e]^3[[s_1^f,s^e],s^e]w^3[w,s^e] ^3 [w,s^e,s^e]s_{r-1}^{2def\beta}\\
&=&(s_1^f)^3 [s_1^f,s^e]^3[[s_1^f,s^e],s^e] s_{r-1}^{2def\beta}\\
&=&  s_{r-1}^{f\gamma+2def\beta},
\end{eqnarray*}
 where the last equality follows from Lemma \ref{relations}(iii).

Therefore for \textbf{R5} to hold we require
$$2de\beta \equiv \gamma (e^{r-2}-1)\pmod 3.$$
We now determine $t^3$ (and this calculation will be used later in a slightly different setting). First of all notice that all the elements of the coset $t\gamma_2(S)$ are $S$-conjugate by Lemma~\ref{conjugate in coset} and $t^3 \in \gamma_1(S) \cap C_S(t)= \gamma_{r-1}(S)$  and so $(t^3)^b= t^3$ for all $b \in S$.  Thus to investigate \textbf{R6}, we may adjust $t$ by conjugacy in $S$ and rather than consider $s^es_1^dw$,  we cube $s^es_1^d$.

\begin{eqnarray*}\label{the calc}
(s^es_1^d)^3&=&s^es_1^ds^es_1^ds^es_1^d\\
&=&(s^e)^2s_1^d[s_1^d,s^e]s^es_1^d[s_1^d,s^e]s_1^d\\
&=&(s^e)^3s_1^d[s_1^d,s^e]^2[s_1^d,s^e,s^e]s_1^d[s_1^d,s^e]s_1^d\\
&=&s_{r-1}^{e\delta} s_1^d[s_1^d,s^e]^2s_1^d[s_1^d,s^e]s_1^d[s_1^d,s^e,s^e]\\
&=&s_{r-1}^{e\delta} s_1^d[s_1^d,s^e]^3 s_1^d  [ s_1^d ,[s_1^d,s^e]]s_1^d [s_1^d,s^e,s^e] \\
&=&s_{r-1}^{e\delta} (s_1^d)^2 [s_1^d,s^e]^3 [s_1^d,s^e,s^e]s_1^d [ s_1^d,[s_1^d,s^e]]\\
&=&s_{r-1}^{e\delta} (s_1^d)^3 [s_1^d,s^e]^3 [s_1^d,s^e,s^e]  [ s_1^d,[s_1^d,s^e]]\\
&=&s_{r-1}^{e\delta}s_{r-1}^{d\gamma}  [ s_1^d,[s_1^d,s^e]]\\
&=&s_{r-1}^{e\delta}s_{r-1}^{d\gamma}s_{r-1}^{ d^2e\beta}= s_{r-1}^{e\delta+d\gamma+e d^2\beta }=s_{r-1}^{e(d^2\beta+\delta)+d\gamma}.
\end{eqnarray*}
 Thus for \textbf{R6} to be satisfied we require
 \begin{equation}\label{e:congruence}
 e(d^2\beta+\delta)+d\gamma\equiv  e^{r-2}f\delta\pmod 3.
 \end{equation}
We have proved
\begin{Prop}\label{autogrps}  The map  $\theta_{e,d,f,v,w}: S\rightarrow S$ is an automorphism of $S=B(r;\beta,\gamma,\delta)$ if and only if the following hold:
\begin{itemize}
\item[(i)]  $ef\beta \equiv e^{r-2}\beta\pmod 3.$
\item[(ii)] $2de\beta \equiv \gamma (e^{r-2}-1)\pmod 3.$
\item[(iii)]  $e(d^2\beta+\delta)+d\gamma\equiv e^{r-2}f\delta\pmod 3$.
\end{itemize}\qed
\end{Prop}

\begin{Prop}\label{autogrpsdesc} The following hold:
\begin{itemize}
\item[(i)] Suppose that $r$ is even.
\begin{enumerate}
\item $\Aut(B(r;0,0,0))= \{\theta_{e,d,f,v,w}\mid v,w,\in \gamma_2(S), d \in \{0,1,2\}, e,f \in \{1,2\}\}$ has order $2^2\cdot 3^{2r-3}$.
 \item    $\Aut(B(r;0,1,0))=\{\theta_{e,0,f,v,w}\mid v,w,\in \gamma_2(S),   e,f \in \{1,2\}\}$ has order $2^2\cdot 3^{2r-4}$.
  \item    $\Aut(B(r;0,2,0))=\{\theta_{e,0,f,v,w}\mid v,w,\in \gamma_2(S), e,f \in \{1,2\}\}$ has order $2^2\cdot 3^{2r-4}$.
\item $\Aut(B(r;0,0,1))=\{\theta_{e,d,e,v,w}\mid v,w,\in \gamma_2(S), d \in \{0,1,2\}, e \in \{1,2\}\}$ has order $2\cdot 3^{2r-3}$.
\item $\Aut(B(r;1,0,0))=\{\theta_{e,0,e,v,w}\mid v,w,\in \gamma_2(S),  e \in \{1,2\}\}$ has order $2\cdot 3^{2r-4}$.
\item $\Aut(B(r;1,0,1))=\{\theta_{e,0,e,v,w}\mid v,w,\in \gamma_2(S), e \in \{1,2\}\}$ has order $2\cdot 3^{2r-4}$.
\item $\Aut(B(r;1,0,2))=\{\theta_{e,0,e,v,w}\mid v,w,\in \gamma_2(S), e \in \{1,2\}\}$ has order $2\cdot 3^{2r-4}$.
\end{enumerate}
\item[(ii)] Suppose that $r$ is odd.
\begin{enumerate}
\item $\Aut(B(r;0,0,0))= \{\theta_{e,d,f,v,w}\mid v,w,\in \gamma_2(S), d \in \{0,1,2\}, e,f \in \{1,2\}\}$ has order $4\cdot 3^{2r-3}$.
\item $\Aut(B(r;0,0,1))=\{\theta_{e,d,1,v,w}\mid v,w,\in \gamma_2(S), d \in \{0,1,2\}, e \in \{1,2\}\}$ has order $2\cdot 3^{2r-3}$.
\item $\Aut(B(r;0,1,0))=\{\theta_{1,0,f,v,w}\mid v,w,\in \gamma_2(S), f \in \{1,2\}\}$ has order $2\cdot 3^{2r-4}$.

\item $\Aut(B(r;1,0,0))=\{\theta_{e,0,1,v,w}\mid v,w,\in \gamma_2(S), e \in \{1,2\}\}$ has order $2\cdot 3^{2r-4}$.
\item $\Aut(B(r;1,0,1))=\{\theta_{e,0,1,v,w}\mid v,w,\in \gamma_2(S),  e \in \{1,2\}\}$ has order $2\cdot 3^{2r-4}$.
\item $\Aut(B(r;1,0,2))=\{\theta_{e,0,1,v,w}\mid v,w,\in \gamma_2(S), e \in \{1,2\}\}$ has order $2\cdot 3^{2r-4}$.
\end{enumerate}
\end{itemize}
\end{Prop}
\begin{proof} We have $\gamma_1(S)= C_S(\gamma_{r-2}(S))$  is characteristic in $S$ by Lemma~\ref{gamma1struct}. Hence every automorphism of $S$ is of the form $\theta_{e,d,f,v,w}$ for suitable $e,d,f,v$ and $w$. The result now follows from Proposition~\ref{autogrps}.  \end{proof}
\begin{Lem}\label{action on sr-1} We have
$s_{r-1}\theta_{e,d,f,v,w}=  s_{r-1}^{e^{r-2}f}$.
\end{Lem}

\begin{proof} Let $\theta=\theta_{e,d,f,v,w}$.   Then, setting $t = s\theta$, $t_1= s_1\theta$ and $t_j=[t_{j-1},t]$ for $j \ge 2$, we have $$t_j=[t_{j-1},t]= [s_{j-1}\theta,s\theta]=[s_{j-1},s]\theta=s_j\theta.$$   Now applying  Equation (\ref{tr-1calc}) yields $s_{r-1}\theta = t_{r-1}= s_{r-1}^{e^{r-2}f}$.
\end{proof}

\section{Fusion systems on maximal class $3$-groups}

Suppose   that $\F$ is a saturated fusion system on $S$, where $S$ is one of the groups $B(r;\beta,\gamma, \delta)$.

\begin{Lem}\label{gamma1essentail} Suppose that $E$ is an $\F$-essential subgroup. If $E \le \gamma_1(S)$, then $E=\gamma_1(S)$ is abelian and $\Aut_\F(E) \cong \SL_2(3)$ or $\GL_2(3)$.
\end{Lem}

\begin{proof} If $\gamma_1(S)$ is abelian, then $E=\gamma_1(S)$ and we have nothing to do.  So suppose that $\gamma_1(S)$ is non-abelian. Since $E$ is $\F$-centric, $E > \gamma_3(S)= Z(\gamma_1(S))$. Suppose that $E \ne \gamma_1(S)$. Then $|\gamma_1(S):E|=3$ and $E$ is abelian.  Furthermore, $E$ is normalized by $\gamma_1(S)$. By Lemma~\ref{2gen} $|E/\Phi(E)|= 9$ and $\Omega_1(\gamma_1(S))= \langle s_{r-1}, s_{r-2}\rangle$.  Since $E$ is $\F$-essential, $[E,\gamma_1(S)]\not \le \Phi(E)$ and so we deduce that $\gamma_{r-1}(S)\not \le \Phi(E)$.  It follows that $E$ is cyclic and therefore $E$ is not essential, a contradiction. Hence $E= \gamma_1(S)$.  Since $|E/\Phi(E)|=9$, we now have  $\Out_\F(E) \cong \SL_2(3)$ or $\Out_\F(E) \cong \GL_2(3)$.  Since the Sylow $2$-subgroup of $\Aut_\F(E)$ has to act faithfully on $\Omega_1(E) = \langle s_{r-1}, s_{r-2}\rangle$, we deduce that $\gamma_{r-1}(S) \ne \gamma_1(S)'$ and this is a contradiction.
\end{proof}

\begin{Lem}\label{otheressentails}
Suppose that $E$ is an $\F$-essential subgroup of $S$ and $E \not \le \gamma_1(S)$. The following hold:
\begin{itemize}
\item[(i)] $E\gamma_2(S)/\gamma_2(S)$ has order $3$;
\item[(ii)] $\Out_\F(E) \cong \SL_2(3)$ or $\Out_\F(E) \cong\GL_2(3)$; and
\item[(iii)] $E$ is either extraspecial of order $3^3$ or elementary abelian of order $3^2$.
\end{itemize}
Furthermore, if $F$ is $\F$-essential with $E\gamma_2(S)= F \gamma_2(S)$, then $E$ and $F$ are $S$-conjugate. In particular, $\F$ has at most four $S$-classes of $\F$-essential subgroups.
\end{Lem}

\begin{proof}  Since $\Phi(S)=\gamma_2(S)$, $E \not \le \gamma_1(S)$ and $E \ne S$, we have $|E\gamma_2(S)/\gamma_2(S)|=3$ which is (i).  Using \cite[Lemma 1.2]{VeraLopez} with Lemma~\ref{gamma1struct} yields that $E\gamma_2(S)$ has maximal class and we know $\gamma_2(S)$ is abelian. We may repeat this argument until we obtain  $E= E\gamma_i(S)$ for some $i \ge 2$ has maximal class. In particular, $|E/\Phi(E)|=9$ and so, as $E$ is $\F$-essential,  either $\Out_\F(E) \cong \GL_2(3)$ or $\Out_\F(E) \cong \SL_2(3)$.  This proves (ii).  Suppose that $|E| \ge 3^4$. We obtain a contradiction by showing that $\Aut_\F(E)$ does not possess a subgroup of order $8$. This follows immediately from Proposition~\ref{autogrpsdesc} when $|E| \ge 3^5$. The case $|E|=3^4$ is a straightforward computation (for example using {\sc Magma} \cite{bosma1997magma}). Hence  $|E| \in \{3^2,3^3\}$  and  $E$ has exponent $3$. This demonstrates   (iii) holds.

Notice that $E \ge Z(S)= \gamma_{r-1}(S) \ge \gamma_1(S)'$  and so $E \cap\gamma_1(S) $ is normalized by $S$. Therefore, as $\Omega_1(\gamma_1(S))=\langle s_{r-2}, s_{r-1}\rangle$ by Lemma~\ref{2gen},  $$E\cap \gamma_1(S)= \begin{cases} \gamma_{r-1}(S) & \mbox{ if $|E|=9$; and} \\  \gamma_{r-2}(S) & \mbox{ if $|E| = 27.$}\end{cases}$$ Suppose $F$ is  an $\F$-essential subgroup with $E < F$. Then $E$ is elementary abelian of order $9$ and $F$ is extraspecial of order $27$.  Since $\Aut_\F(F)$ acts transitively on the maximal subgroups of $F$, $E$ is $\Aut_\F(F)$-conjugate to $\gamma_{r-2}(S)$. This contradicts the fact that $E$ is fully $\F$-normalized and we conclude that there is no such containment.

 Suppose that $E\gamma_2(S)= F\gamma_2(S)$. We may assume that $E \cap \gamma_1(S) \le F \cap \gamma_1(S)$. Let $x \in E\setminus \gamma_2(S)$ and $y \in F \setminus \gamma_2(S)$ be such that $x\gamma_2(S)= y\gamma_2(S)$.  Then by Lemma~\ref{conjugate in coset}, $x$ and $y$ are $S$-conjugate. Hence we may suppose that $x \in E \cap F$.  Then $E= \langle x \rangle (E \cap \gamma_2(S))$ and $F=  \langle x \rangle (F \cap \gamma_2(S))$. Hence $E \le F$ and thus $E=F$ as claimed.
\end{proof}

For $x \in S \setminus  \gamma_1(S)$ to be contained in an elementary abelian subgroup of order $9$ or an extraspecial subgroup of order $27$, it suffices that $x$ has order $3$ since then $\langle x\rangle  \gamma_{r-1}(S)$
and $\langle x\rangle  \gamma_{r-2}(S)$ are such subgroups. Using Equation (\ref{e:congruence}) we see that $ss_1^d$ has order $3$ if and only if   $$d^2\beta+\delta +d\gamma \equiv 0 \pmod 3.$$ Table \ref{tableelts3} lists the groups $B(r;\beta,\gamma,\delta)$ and cosets $ss_1^d\gamma_2(S)$ which consist of elements of order $3$.

\begin{table}[H]
\renewcommand{\arraystretch}{1.4}

\begin{tabular}{|cccc||cccc|}
\hline
$S$&$s\gamma_2(S)$& $ss_1\gamma_2(S)$&$ss_1^2\gamma_2(S)$&$S$&$s\gamma_2(S)$& $ss_1\gamma_2(S)$&$ss_1^2\gamma_2(S)$\\
\hline
 $B(r;0,0,0)$, $r$ even&\checkmark&\checkmark&\checkmark& $B(r;0,0,0)$, $r$ odd&\checkmark&\checkmark&\checkmark\\
$B(r;0,1,0)$, $r$ even & \checkmark&& &$B(r;0,1,0)$, $r$ odd&\checkmark& & \\
$B(r;0,2,0)$, $r$ even & \checkmark&& &$B(r;0,0,1)$, $r$ odd& & &  \\
 $B(r;0,0,1)$, $r$ even & &&&$B(r;1,0,0)$, $r$ odd&\checkmark& & \\
 $B(r;1,0,0)$, $r$ even &\checkmark& & &  $B(r;1,0,1)$, $r$ odd& & & \\
$B(r;1,0,1)$, $r$ even & & & &$B(r;1,0,2)$, $r$ odd& &\checkmark&\checkmark \\
$B(r;1,0,2)$, $r$ even & &\checkmark&\checkmark&&&&\\
\hline
\end{tabular}\caption{Elements of order $3$ in the designated cosets of $S=B(r;\beta,\gamma,\delta)$}
\label{tableelts3}
\end{table}

At this stage, we can confirm that if $\beta=0$, then all the potential fusion systems have been discovered by   \cite[Theorem~5.10]{DiazRuizViruel2007}.

 \begin{Lem}\label{restrictautos}Suppose that $E$ is $\F$-essential. Then $$N_{\Aut_\F(E)} (\Aut_{S}(E))= \{\beta|_{E} \mid \beta \in \Aut_\F(N_S(E)), E\beta= E\}=\{\alpha|_E\mid \alpha \in \Aut_\F(S), E\alpha =E\}.$$
 \end{Lem}

\begin{proof}   Recall the definition of $H_E$ from \cite[Proposition I.3.3]{AschbacherKessarOliver2011}. Then, as $E$ is $\F$-essential, $H_E/\Inn(E)$ is strongly $p$-embedded in $\Out_\F(E)$. Since, by Lemmas~\ref{gamma1essentail} and \ref{otheressentails} (ii),$ \Out_\F(E) \cong \SL_2(3)$ or $\GL_2(3)$, $H_E= N_{\Aut_\F(E)} (\Aut_{S}(E))$. Let $\theta$ be a generator of $H_E$. Then, by definition, there exists $R >E$ and $\psi \in \Hom_\F(R,S)$ with $E\psi = E$ such that $\theta=\psi|_E$. Thus $\psi|_{N_R(E)}\in \Hom(N_R(E), S)$. Since $|N_S(E)/E|=3$, $N_R(E)= N_S(E)$ and so $\psi|_{N_S(E)}\in \Aut_\F(N_S(E))$ and $\theta=\psi|_E$. It follows that $$H_E= \{\psi |_E\mid \psi \in \Aut_\F(N_S(E)), E\psi= E\}$$ and so the first equality holds.

By Lemmas~\ref{gamma1essentail} and \ref{otheressentails}, no $\F$-essential subgroup properly contains $E$, so   every element of $\Aut_\F(N_S(E))$ is the restriction of an element of $\Aut_\F(S)$ by Alperin's Theorem \cite[Theorem I.3.5]{AschbacherKessarOliver2011}. This provides the second asserted equality.
\end{proof}

 \begin{Lem}\label{gamma1essentialgroups} Assume that $\gamma_1(S)$ is $\F$-essential.
Then $r$ is odd and $(\beta, \gamma,\delta) \in \{(0,0,0),(0,1,0)  \}$.
 \end{Lem}

\begin{proof} Let $A= \gamma_1(S)$ and assume that $A$ is $\F$-essential.  Then $A$ is abelian and $\Aut_\F(A)$ contains a normal subgroup isomorphic to $\SL_2(3)$ by Lemma~\ref{gamma1essentail}. We conclude that all the elements of $A\setminus \Phi(A)$ have the same order and so $|A|= 3^{2k}$ for some $k \ge 2$ and $r$ is odd.
Let $\tau_A $ be an element of order $2$ in $\Aut_\F(A)$ which corresponds to the centre of $\SL_2(3)$.  Then $\tau_A= \sigma|_A$ for some $\sigma \in \Aut_\F(S)$ and $\sigma$ centralizes $S/A$ and inverts $A/\gamma_2(S)$.  Now consulting Proposition~\ref{autogrpsdesc} delivers the conclusion $(\beta, \gamma,\delta) \in \{(0,0,0),(0,1,0) \}$.
\end{proof}

\begin{Lem} Let $\F$ be a saturated fusion system on $B(r;0,\gamma,\delta)$ with at least one class of $\F$-essential subgroups. Then $\F$ is as described in \cite[Theorem~5.10]{DiazRuizViruel2007}.
\end{Lem}

\begin{proof}
Lemma~\ref{otheressentails} and Table~\ref{tableelts3} indicate that the groups $B(r;0,0,1)$ have no $\F$-essential subgroups which are not contained in $\gamma_1(S)$. This shows that all candidates for $S$ have been considered in \cite[Theorem 5.10]{DiazRuizViruel2007}. Here we note that Table~\ref{tableelts3} indicates that $E_0$ and $V_0$ have exponent $3$ when $S=B(r;0,1,0)$. If $r$ is even, we obtain the examples in \cite[Theorem~5.10]{DiazRuizViruel2007} whereas when $r$ is odd there are no examples. To understand this observe that  $E_0\gamma_2(S)=V_0\gamma_2(S) = \langle s \rangle \gamma_2(S)$. If $E$ is one of  $E_0$ or $V_0$ and is $\F$-essential, then $\Out_\F(E) $ contains a subgroup isomorphic to $\SL_2(3)$ and hence an element $\sigma $ of order $2$ which inverts $E/\Phi(E)$ and so also  $\langle s\rangle \Phi(E)/\Phi(E)$. By Lemma~\ref{restrictautos} $\sigma$ is the restriction of an automorphism $\sigma^*$ of $S$.  However Proposition~\ref{autogrpsdesc} (ii)(3) shows that every automorphism of $S$ which normalizes $E\gamma_2(S)$ actually centralizes $\langle s \rangle \gamma_2(S)/\gamma_2(S)$ and so we have no candidates for $\sigma^*$. Thus we recover the results from \cite[Theorem~5.10]{DiazRuizViruel2007} in this case.
\end{proof}

From now on we assume that $\beta\ne 0$. For $d \in \{0,1,2\}$ define
$$V_d=\langle ss_1^d, s_{r-1}\rangle \hspace{2mm} \mbox{ and } \hspace{2mm} E_d =\langle ss_1^d, s_{r-1},s_{r-2}\rangle.$$ Hence  $V_d$ is abelian of order $9$ and $E_d$ is extraspecial  of order 27.  In addition, $V_d$ and $E_d$ have exponent $3$ if and only if $ss_1^d$ has order $3$. Thus using Table~\ref{tableelts3} we obtain the following table of possible $\F$-essential subgroups up to $S$-conjugacy.

\begin{table}[H]
\renewcommand{\arraystretch}{1.4}
\begin{tabular}{|ccccccc|}
\hline
$S$&$E_0$& $E_1$&$E_2$&$V_0$&$V_1$& $V_2$ \\
\hline
 $B(r;1,0,0)$   &\checkmark& & & \checkmark &&\\
$B(r;1,0,1)$ &    & & & && \\
$B(r;1,0,2)$ &  &\checkmark&\checkmark& &\checkmark&\checkmark\\
\hline
\end{tabular}\caption{ Candidates for the $\F$-essential subgroups} \label{tableelts3b=1}
\end{table}
We record the following result:
\begin{Lem}\label{(101)} $B(r;1,0,1)$ is resistant.
\end{Lem}

\begin{proof}  By Lemma \ref{gamma1essentialgroups}, $\gamma_1(S)$ is not $\F$-essential and there are no other candidates for essential subgroups by Table \ref{tableelts3b=1}.
\end{proof}

\begin{Lem}\label{(1,0,0)} Suppose that $S=B(r;1,0,0)$ and $\F$ is a  saturated fusion system on $S$ which has at least one $\F$-conjugacy class of $\F$-essential subgroups. Then  $V_0$ represents the unique such class, $\Aut_\F(V_0) \cong \SL_2(3)$, $|\Out_\F(S)|=2$ and either
\begin{enumerate}
    \item[(i)] $r$ is even and $\F$ is reduced; or
    \item[(ii)] $r=2k+1$ is odd, and $O^3(\F)$ is a subsystem of index $3$ isomorphic to the fusion system of $\PSL_3(q)$ for some prime power $q$ with $v_3(q-1)=k$.
        \end{enumerate}
\end{Lem}

\begin{proof} Let $D$ be an $\F$-essential subgroup.  Using Table~\ref{tableelts3b=1} we see that up to $S$-conjugacy $D= V_0$ or $D= E_0$.  By Lemma~\ref{restrictautos}, $|\Out_\F(S)| \ge 2$ so we conclude from Proposition~\ref{autogrpsdesc} that $|\Out_\F(S)|=2$. Thus without loss of generality, we define $\theta \in \Aut_\F(S)$ via:
 $$\theta :=\begin{cases} \theta_{2,0,2,1,1} & \mbox{ if $r$ is even; } \\ \theta_{2,0,1,1,1} & \mbox{ if $r$ is odd. } \end{cases}$$
 We see that $\theta$ inverts $s$ and
 from Lemma~\ref{action on sr-1}, $$s_{r-1}\theta = \begin{cases} s_{r-1}\theta_{2,0,2,1,1} = s_{r-1}^{2^{r-2}.2}=s_{r-1}^{-1} & \mbox{ if $r$ is even; } \\
s_{r-1}\theta_{2,0,1,1,1} = s_{r-1}^{2^{r-2}}=s_{r-1}^{-1} & \mbox{ if $r$ is odd. } \end{cases}
  $$
 By Lemma~\ref{restrictautos} we have $\Out_\F(D)\cong \SL_2(3)$ and the central involution in $\Out_\F(D)$ is the image of $\theta|_D$.
If $D = E_0$, then $\theta|_D$ centralizes $Z(D) = \gamma_{r-1}(S)$ so we conclude that $D=V_0$. Now $N_S(D)\langle \theta|_{N_S(D)} \rangle$ is isomorphic to the normalizer of the Sylow $3$-subgroup $M$ of a group $P$ isomorphic with $3^2:\SL_2(3)$ and so the amalgamated product $G=P*_MS\langle \theta\rangle$ realizes $\F=\F_S(G)$. Since $D$ is minimal among all $\F$-centric subgroups of $S$, $\F$ is saturated by  \cite[Theorem C]{Semeraro2014}.

If $r$ is even then $S=[S,\theta]$ so $\mathfrak{foc}(\F)=S$ and $\F$ is reduced as $\theta|_D \in O^{3'}(\Aut_\F(D))$. If $r$ is odd, then, since $\theta$ centralizes $s_1$, $$[S,\theta]=\langle s,\gamma_2(S)\rangle = \langle s, s_2,s_3, \ldots, s_{r-1} \rangle \cong B(r-1;0,0,0).$$ Since $D$, $D^{s_1}$ and $D^{s_1^2}$ are $O^3(\F)$-essential and not fused in $O^3(\F)$, the result follows from \cite[Tables 2,4] {DiazRuizViruel2007}.
\end{proof}

\begin{Lem}\label{(1,0,2)} Suppose that $S=B(r;1,0,2)$ and $\F$ is a  saturated fusion system on $S$ which has at least one $\F$-conjugacy class of $\F$-essential subgroups.  Then  $r$ is even and one of the following holds:
\begin{enumerate}
 \item[(i)] $V_1 $  represents the unique   $\F$-conjugacy class of $\F$-essential subgroups, $\Aut_\F(V_1) \cong \SL_2(3)$, $|\Out_\F(S)|=2$ and $\F$ is reduced;
 \item [(ii)]$V_2 $ represents the unique   $\F$-conjugacy class of $\F$-essential subgroups, $\Aut_\F(V_2 ) \cong \SL_2(3)$, $|\Out_\F(S)|=2$ and $\F$ is reduced; or
  \item [(iii)] there are two $\F$-conjugacy classes of $\F$-essential subgroups represented by $V_1$ and $V_2$ with $\Aut_\F(V_1 ) \cong \Aut_\F(V_2) \cong \SL_2(3)$, $|\Out_\F(S)|=2$ and $\F$ is reduced.
 \end{enumerate}
\end{Lem}

\begin{proof}
 Let $D$ be an $\F$-essential subgroup. By Table~\ref{tableelts3b=1}  we have
$D\in\{V_1, V_2,E_1,E_2\}$ up to   $S$-conjugacy. Arguing as in Lemma \ref{(1,0,0)} (using Lemma~\ref{restrictautos} and Proposition~\ref{autogrpsdesc}) we see that $|\Out_\F(S)|=2$. If $r$ is odd then $\Aut_\F(S)$ permutes $\{\langle ss_1\rangle\gamma_2(S),\langle ss_1^2\rangle\gamma_2(S)\}$ transitively by Proposition \ref{autogrpsdesc}. In particular no element of $\Aut_\F(S)$ of order $2$ normalizes an element of $\{V_1, V_2,E_1,E_2\}$. Using Lemma~\ref{restrictautos} we deduce that $r$ is even. Set $$\theta := \begin{cases} \theta_{2,0,2,1,s_2s_{r-1}^{-1}} &\text{if } D \in \{V_1,E_1\}\\ \theta_{2,0,2,s_2^2,1}&\text{if } D \in \{V_2,E_2\}.
 \end{cases}$$
 Then $s_{r-1}\theta= s_{r-1}^{-1}$ and so as $\Out_\F(D)\cong\SL_2(3)$ we see as before that $D\in \{V_1,V_2\}$. Observe that  if $D= V_1$, then
$$(ss_1)\theta = s\theta s_1 \theta = s^2 s_1^2 s_2 s_{r-1}^{-1} = (ss_1)^{-1}$$ and, if $D=V_2$, then
 $$(ss_1^2)\theta= s\theta (s_1 \theta)^2 = s^2 s_2^2 s_1^4 = (ss_1^2)^{-1}.$$  Thus $\theta$ inverts $D$ as required.

Again observe  that   $S=[S,\theta]$ so $\mathfrak{foc}(\F)=S$ and $\F$ is reduced as $\theta|_D \in O^{3'}(\Aut_\F(D))$.
For $i=1,2$, $N_S(V_i)\langle \theta|_{N_S(V_i)} \rangle$ is isomorphic to the normalizer of the Sylow $3$-subgroup $M_i$ of a group $P_i$ isomorphic with $3^2:\SL_2(3)$. Hence there is an amalgamated product $G_i=P_i*_{M_i}S\langle \theta\rangle$ which realizes a fusion system $\F_i=\F_S(G_i)$ satisfying the conditions in (i) and (ii) respectively. Since $V_i$ is minimal among all $\F_i$-centric subgroups of $S$, $\F_i$ is saturated by  \cite[Theorem C]{Semeraro2014}. In particular there are unique fusion systems satisfying these conditions.  Now, since $V_2$ is a fully $\F_1$-normalized subgroup which is minimal among all $\F_1$-centric subgroups of $S$, \cite[Theorem C]{Semeraro2014} also implies that $\langle \F_1,\F_2 \rangle$ is saturated. This is the unique fusion system described by (iii).
\end{proof}

\begin{Lem}\label{they're exotic} The fusion systems described in Lemmas~\ref{(1,0,0)} and \ref{(1,0,2)} are exotic.
\end{Lem}

\begin{proof} Suppose that $\F$ represents one of the fusion systems of interest and let  $S= B(r; 1,0,0)$ or $B(r;1,0,2)$.  If $C$ is a non-trivial strongly $\F$-closed subgroup of $S$, then $C$ is normal in $S$ and so $s_{r-1} \in C$. Thus $C \cap V \ne 1$ where $V$ is an $\F$-essential subgroup. It follows that $V \le C$ and then $\langle V^S\rangle= V\gamma_2(S) \le C$. In fact $V\gamma_2(S)$ is strongly $\F$-closed if $S= B(r;1,0,0)$ or $S= B(r;1,0,2)$ and $\F$ has only one $\F$-class of $\F$-essential subgroups. Since $\langle V^S\rangle\cong B(r-1;0,0,0)$, \cite[Proposition 2.19]{DiazRuizViruel2007} applies to say that if $\F$ is realised by a finite group $G$, then it is realized by an almost simple group.  Now  the arguments in \cite[page 1751 (a), (b) and (c)]{DiazRuizViruel2007} prove the result.
\end{proof}

\bibliographystyle{amsalpha}
\def\cprime{$'$}
\providecommand{\bysame}{\leavevmode\hbox to3em{\hrulefill}\thinspace}
\providecommand{\MR}{\relax\ifhmode\unskip\space\fi MR }
\providecommand{\MRhref}[2]{%
  \href{http://www.ams.org/mathscinet-getitem?mr=#1}{#2}
}
\providecommand{\href}[2]{#2}

\end{document}